\documentclass[notoc]{cimart}

\usepackage{listings}

\title{On the sizes of the maximal prime power divisors of factorials}

\authors{Dan Levy}

\authorinfo{The Academic College of Tel-Aviv-Yaffo, Israel}{danlevy@mta.ac.il}

\abstract{%
    Let $p$ be any prime, and $p^{\nu_p(n!)}$ the maximal power of $p$ dividing $n!$. We prove that there exists a positive integer $n_0$, which depends only on $p$, such that $q^{\nu_q(n!)} < p^{\nu_p(n!)}$ for all $n \ge n_0$ and all primes $q > p$. For twin primes $p$ and $q=p+2$, we prove that the minimal $n_0$ satisfying $q^{\nu_q(n!)} < p^{\nu_p(n!)}$ for all $n \ge n_0$
    is given by $n_0 = \frac{p^{2}+p}{2}$.
    }



\VOLUME{34}
\ISSUE{1}
\NUMBER{11}
\DOI{https://doi.org/10.46298/cm.17290}
\licence{CC BY-SA 4.0}
\editinfo{January 13, 2026}{June 20, 2026}{Eric Swartz}
\acknowledgments{I am grateful to Carlo Sanna for informing me about Hassani's results.}

\begin{document}

\section*{Introduction}

Finding the prime factorization of an arbitrary integer is believed to be a
difficult algorithmic problem, although its precise complexity is not known.
However, for the special case of factorials, $n!:=n\cdot \left( n-1\right)
\cdots 2\cdot 1$, $n\geq 2$, we have the following complete and
elegant solution:%
\begin{equation}
n!=\prod\limits_{p\leq n}p^{\nu _{p}\left( n!\right) },
\label{Eq_n!_prime_factorization}
\end{equation}%
where the product is over all primes $p$ between $2$ and $n$ and the
multiplicity of $p$, denoted $\nu _{p}\left( n!\right) $, is given by the
Legendre formula \cite[Theorem (6-9)]{BurtonDavidBook}:%
\begin{equation}
\nu _{p}\left( n!\right) =\sum\limits_{j=1}^{\infty }\left\lfloor \frac{n}{%
p^{j}}\right\rfloor =\sum\limits_{j=1}^{L}\left\lfloor \frac{n}{p^{j}}%
\right\rfloor \text{; }L:=\left\lfloor \log _{p}n\right\rfloor .
\label{Legendre_formula}
\end{equation}%
An alternative form of this formula, which will be our main working tool, is
(see \cite[Exercise~7 in Problems for Section~6.3]{BurtonDavidBook}):
\begin{equation}
\nu _{p}\left( n!\right) =\frac{n-s_{p}\left( n\right) }{p-1}\text{,}
\label{Legendre_formula_alternative}
\end{equation}%
where $s_{p}\left( n\right) $ denotes the sum of the base $p$ digits of $n$
(see first paragraph of Section \ref{Subsect_twin_primes_proof}).

Here we study the relative sizes of the maximal prime power divisors $p^{\nu
_{p}\left( n!\right) }$ of $n!$. We will prove:

\begin{theorem}
\label{Th_p<q}Let $p$ be any prime. Then there exists some $n_{0}\left(
p\right) \in \mathbb{N}$ such that 
\begin{equation*}
q^{\nu _{q}\left( n!\right) }<p^{\nu _{p}\left( n!\right) }\text{; }\forall
n\geq n_{0}\left( p\right) \text{ and for all primes }q>p.
\end{equation*}
\end{theorem}

\begin{corollary}
\label{Coro_p_1<...<p_m}Let $m$ be any positive integer and let $%
p_{1}<p_{2}<\cdots <p_{m}$ be the first $m$ prime
numbers ordered increasingly, so that $p_1 = 2$, $p_2 = 3$, $p_3 = 5$, and so on. Then there exists some $n_{0}\left( m\right)
\in \mathbb{N}$ such that 
\begin{equation*}
2^{\nu _{2}\left( n!\right) }>3^{\nu _{3}\left( n!\right) }>5^{\nu
_{5}\left( n!\right) }>\cdots >p_{m}^{\nu _{p_{m}}\left( n!\right) }>q^{\nu
_{q}\left( n!\right) }\text{; }\forall n\geq n_{0}\left( m\right)
\end{equation*}%
for any prime $q$ which is larger than $p_{m}$.
\end{corollary}

\begin{proof}
This follows from Theorem \ref{Th_p<q} by taking $n_{0}\left( m\right) =\max
\left\{ n_{0}\left( p_{1}\right) ,\ldots ,n_{0}\left( p_{m}\right) \right\} $%
.
\end{proof}

\begin{corollary}
\label{Coro_2_domination}Let $q$ be any odd prime, and $n\geq 2$ any
integer. Then $2^{\nu _{2}\left( n!\right) }>q^{\nu _{q}\left( n!\right) }$
with the single exception $q=n=3$.
\end{corollary}

Shortly after uploading the first version (arXiv:submit/7138707) of the
present paper, I have learned about a paper by Mehdi Hassani \cite%
{HassaniMhedi2005}, with results similar to Theorem \ref{Th_p<q} and
Corollary \ref{Coro_2_domination}.

\begin{theorem}[see {\cite[Theorem~5.2]{HassaniMhedi2005}}]
Suppose $p$ and $q$ are primes such that $p<q$. Then for sufficiently large values of $n$, we have
\begin{equation*}
p^{\nu _{p}\left( n!\right) }>q^{\nu _{q}\left( n!\right) }\text{. }
\end{equation*}
\end{theorem}

\begin{corollary}[see {\cite[Corollary~5.3]{HassaniMhedi2005}}]
For $n=2$ and $n\geq 4$ we have 
\begin{equation*}
2^{\nu _{2}\left( n!\right) }>3^{\nu _{3}\left( n!\right) }\text{. }
\end{equation*}
\end{corollary}

Hassani mentions that the claim of the last corollary was already proved in 
\cite{Balacenoiu1997}.

It is clear that Theorem \ref{Th_p<q} implies Hassani's Theorem 5.2.
Furthermore, the formulation of the conclusion of Theorem \ref{Th_p<q}
refines the \textquotedblleft sufficiently large $n$\textquotedblright\
condition of Hassani's statement. A priori, deciding whether $n$ is
sufficiently large may depend on both $p$ and $q$, but the proof of Theorem %
\ref{Th_p<q} reveals that it can be made to depend on $p$ only. Fix any
prime $p$ and denote by $p_{\text{succ}}$ its prime successor, that is, $p_{%
\text{succ}}$ is the smallest prime which is strictly larger than $p$. The
proof shows that any positive integer solution $n$ of the inequality%
\begin{equation*}
\left\lfloor \log _{p}n\right\rfloor <\left( n-1\right) h_{p}\left( p_{\text{%
succ}}\right) -\frac{p-2}{p-1}\text{,}
\end{equation*}%
where $h_{p}\left( x\right) :=\frac{1}{p-1}-\frac{\log _{p}x}{x-1}$,
satisfies $q^{\nu _{q}\left( n!\right) }<p^{\nu _{p}\left( n!\right) }$ for
any prime $q>p$. Moreover, there exists $n_{0}\in \mathbb{N}$ such that the
above inequality holds true for all $n\geq n_{0}$. Now, since $p_{\text{succ}%
}$ is uniquely determined by $p$, the minimal such $n_{0}$ is a function of $%
p$. Denoting this minimal value by $n_{0}\left( p\right) $, we get that $%
q^{\nu _{q}\left( n!\right) }<p^{\nu _{p}\left( n!\right) }$ holds true for
all $n\geq n_{0}\left( p\right) $ and all $q>p$. This difference is
significant as is already evident when one compares between Hassani's
Corollary 5.3 and Corollary \ref{Coro_2_domination}, and it gains even more
strength when considering the second main result of the current paper, to
which we now turn.

The existence of $n_{0}\left( p\right) $ implies the existence of a minimal
positive integer $$n_{\min }\left( p\right) \leq n_{0}\left( p\right) $$ such
that $$p_{\text{succ}}^{\nu _{p_{\text{succ}}}\left( n!\right) }<p^{\nu
_{p}\left( n!\right) }$$ for all $n\geq $ $n_{\min }\left( p\right) $. While
it is evident that $n_{\min }\left( p\right) \leq n_{0}\left( p\right) $,
this bound need not be tight. For instance, if $p=2$, the proof of Corollary %
\ref{Coro_2_domination} gives $n_{0}\left( 2\right) =21$, while $n_{\min
}\left( 2\right) =4$. The second result of the paper gives the exact value
of $n_{\min }\left( p\right) $ in the special case $p_{\text{succ}}=p+2$,
(equivalently, $p$ and $p_{\text{succ}}$ are twin primes).

\begin{theorem}
\label{Th_twin_primes} Let $p$ and $q=p+2$ be prime twins. Then%
\begin{equation*}
n_{\min }\left( p\right) =\frac{p^{2}+p}{2}.
\end{equation*}
\end{theorem}

Section \ref{Sect_Proofs} is devoted to the proofs of Theorem \ref{Th_p<q},
Corollary \ref{Coro_2_domination} and Theorem \ref{Th_twin_primes}. The
arguments combine the use of (\ref{Legendre_formula_alternative}), some
guess work based on numerical examples that were studied using the computer
algebra system GAP \cite{GAP2025} (see Section \ref{Sect_GAP_Code} for a
sample code), and standard manipulations of inequalities that also rely on
elementary results from the calculus of real one variable functions.

It seems appropriate to conclude the introduction with the following two
remarks. The initial motivation behind the question addressed by Theorem \ref%
{Th_p<q} arose in the context of finite group theory, where the maximal
prime power divisors of $n!$ are the orders of the Sylow subgroups of the
finite Symmetric groups $S_{n}$. Here the author is happy to embrace the
following sentence, quoted from \cite{DiaconisEtAl2025}: \textquotedblleft
Moreover, any simple question about $S_{n}$ is worth
studying(!).\textquotedblright\ The second remark concerns the simple
looking answer to the question addressed by Theorem \ref{Th_twin_primes}.
One wonders if this result can be generalized to other primes (with $p_{%
\text{succ}}-p>2$), and if so, can $n_{\min }\left( p\right) $ be written in
polynomial form whose coefficients are parametrized by $p_{\text{succ}}-p$?

\section{Proofs\label{Sect_Proofs}}

\subsection{Proof of Theorem~\ref{Th_p<q}}

\begin{lemma}
\label{Lem_from_powers_to_logs}Let $p<q\leq n$, where $p$ and $q$ are
primes and $n$ is an integer. Then 
\begin{equation*}
q^{\nu _{q}\left( n!\right) }<p^{\nu _{p}\left( n!\right)
}\Longleftrightarrow \left( \log _{p}q\right) \cdot \frac{n-s_{q}\left(
n\right) }{q-1}<\frac{n-s_{p}\left( n\right) }{p-1}.
\end{equation*}
\end{lemma}

\begin{proof}
Using the identity $q=p^{\log _{p}q}$ we have 
\begin{align*}
q^{\nu _{q}\left( n!\right) }<p^{\nu _{p}\left( n!\right) }
&\Longleftrightarrow 
\left( p^{\log _{p}q}\right) ^{\nu _{q}\left( n!\right) }<p^{\nu _{p}\left(
n!\right) } \\
&\Longleftrightarrow
p^{\left( \log _{p}q\right) \cdot \nu _{q}\left( n!\right) }<p^{\nu
_{p}\left( n!\right) } \\
&\Longleftrightarrow
\left( \log _{p}q\right) \cdot \nu _{q}\left( n!\right) <\nu _{p}\left(
n!\right) .
\end{align*}
Substituting $\nu _{q}\left( n!\right) $ and $\nu _{p}\left( n!\right) $
from (\ref{Legendre_formula_alternative}) in the last inequality, yields the
claim of the lemma.
\end{proof}

\begin{lemma}
\label{Lem_h_p(x)}Let $p$ be a prime and set 
\begin{equation}
h_{p}\left( x\right) :=\frac{1}{p-1}-\frac{\log _{p}x}{x-1}=\frac{1}{p-1}-%
\frac{1}{\ln p}\frac{\ln x}{x-1}\text{; }\forall x\in \lbrack p,\infty )%
.  \label{Def_h_p(x)}
\end{equation}%
Then $h_{p}\left( x\right) $ is differentiable and monotonically increasing
in its domain of definition $[p,\infty )$. It follows that for any prime $%
q>p $ we have 
\begin{equation*}
h_{p}\left( q\right) \geq h_{p}\left( p_{\text{succ}}\right) >0\text{.\label%
{Inequality_h_p(q)}}
\end{equation*}
\end{lemma}

\begin{proof}
The differentiability of $h_{p}\left( x\right) $ in $[p,\infty )$ is clear.
In order to prove that it is monotonically increasing, we compute its first
derivative in $[p,\infty )$%
\begin{equation*}
h_{p}^{\prime }\left( x\right) =\frac{1}{\ln p}\frac{1}{\left( x-1\right)
^{2}}\left( \ln x+\frac{1}{x}-1\right) \text{,}
\end{equation*}%
and check that $h_{p}^{\prime }\left( x\right) $ is positive in $[p,\infty )$%
. Clearly, for any $x\in \lbrack p,\infty )$ we have $h_{p}^{\prime }\left(
x\right) >0$ if and only if $$g\left( x\right) :=\ln x+\frac{1}{x}-1>0.$$ For $p=2$ we
have $$g\left( 2\right) =\ln 2+\frac{1}{2}-1\geq 0.693-\frac{1}{2}>0.$$ For 
$p\geq 3$ we have $\ln p>1$ and hence $$g\left( p\right) =\ln p+\frac{1}{p}-1>
\frac{1}{p}>0.$$ Thus $g\left( p\right) >0$ for all primes $p$. Now $$
g^{\prime }\left( x\right) =\frac{1}{x}-\frac{1}{x^{2}}$$ and $g^{\prime
}\left( x\right) >0$ for all $x\in \lbrack p,\infty )$ since $p>1$. It
follows that $g\left( x\right) $ is monotonically increasing in $[p,\infty )$, which, together with $g\left( p\right) >0$, shows that $g\left( x\right) >0$
for all $x\in \lbrack p,\infty )$. Hence $h_{p}^{\prime }\left( x\right) >0$
for all $x\in \lbrack p,\infty )$, and so $h_{p}\left( x\right) $ is
monotonically increasing in $[p,\infty )$. This implies that for any $u>p$,
the function $h_{p}\left( x\right) $ has a global minimum in $[u,\infty )$,
given by $h_{p}\left( u\right) $. The value $h_{p}\left( u\right) $ is
strictly positive since $h_{p}\left( p\right) =0$, and $h_{p}\left( x\right) 
$ is monotonically increasing in $[p,\infty )$. Choosing $u=p_{\text{succ}}$
concludes the proof.
\end{proof}

\begin{proof}[Proof of Theorem~\ref{Th_p<q}]
Let $p$ be any prime. We have to prove the existence of $n_{0}\in \mathbb{N}$%
, which may depend on $p$, such that 
\begin{equation*}
q^{\nu _{q}\left( n!\right) }<p^{\nu _{p}\left( n!\right) }\text{; }\forall
n\geq n_{0}\text{ and for all primes }q>p\text{ .}
\end{equation*}%
Since $q^{\nu _{q}\left( n!\right) }\geq 1$ we must have $n_{0}\geq p$.
Under this assumption $\nu _{p}\left( n!\right) >0$ for all $n\geq n_{0}$
and hence $q^{\nu _{q}\left( n!\right) }<p^{\nu _{p}\left( n!\right) }$ is
trivially true for all primes $q>p$ and all $n\geq n_{0}$ with $q>n$. Hence
it suffices to show that there exists $n_{0}\in \mathbb{N}$ such that $%
n_{0}>p $ and $q^{\nu _{q}\left( n!\right) }<p^{\nu _{p}\left( n!\right) }$
for all $n\geq n_{0}$ and for all primes $q$ with $p<q\leq n$.
By Lemma \ref{Lem_from_powers_to_logs} we have to prove the existence of $%
n_{0}\in \mathbb{N}$ such that \ \ \ \ 
\begin{equation}
\left( \log _{p}q\right) \cdot \frac{n-s_{q}\left( n\right) }{q-1}<\frac{%
n-s_{p}\left( n\right) }{p-1}  \label{Eq_1}
\end{equation}%
for all $n\geq n_{0}$ and for all primes $q$ such that $p<q\leq n$.
Since $n$ is positive, we have $s_{q}\left( n\right) \geq 1$. This implies 
\begin{equation}
\left( \log _{p}q\right) \cdot \frac{n-s_{q}\left( n\right) }{q-1}\leq
\left( \log _{p}q\right) \cdot \frac{n-1}{q-1}.  \label{Eq_2}
\end{equation}%
On the other hand, the number of base $p$ digits of $n$ is $\left\lfloor
\log _{p}n\right\rfloor +1$, and the largest digit of $n$ in base $p$ is $%
p-1 $. Therefore 
\begin{equation*}
s_{p}\left( n\right) \leq \left( p-1\right) \left( \left\lfloor \log
_{p}n\right\rfloor +1\right) .
\end{equation*}%
This implies 
\begin{equation}
\frac{n-\left( p-1\right) \left( \left\lfloor \log _{p}n\right\rfloor
+1\right) }{p-1}\leq \frac{n-s_{p}\left( n\right) }{p-1}.
\label{Eq_3}
\end{equation}%
Combining (\ref{Eq_2}) and (\ref{Eq_3}) it follows that any prime $q$ and
integer $n$ satisfying 
\begin{equation}
\left( \log _{p}q\right) \cdot \frac{n-1}{q-1}<\frac{n-\left( p-1\right)
\left( \left\lfloor \log _{p}n\right\rfloor +1\right) }{p-1}\text{,}
\label{Eq_4}
\end{equation}%
and $p<q\leq n$, also satisfy (\ref{Eq_1}). Hence our aim now is to prove
the existence of $n_{0}\in \mathbb{N}$ such that (\ref{Eq_4}) holds true for
all $n\geq n_{0}$ and for all primes $q$ such that $p<q\leq n$. We start
with the following equivalences:%
\begin{gather*}
\left( \log _{p}q\right) \cdot \frac{n-1}{q-1}<\frac{n-\left( p-1\right)
\left( \left\lfloor \log _{p}n\right\rfloor +1\right) }{p-1} \\
\Longleftrightarrow \\
n\frac{\left( \log _{p}q\right) }{q-1}-\frac{\left( \log _{p}q\right) }{q-1}<%
\frac{n}{p-1}-\left( \left\lfloor \log _{p}n\right\rfloor +1\right) \\
\Longleftrightarrow
\end{gather*}%
\begin{equation}
\left\lfloor \log _{p}n\right\rfloor <n\left( \frac{1}{p-1}-\frac{\left(
\log _{p}q\right) }{q-1}\right) -\left( 1-\frac{\left( \log _{p}q\right) }{%
q-1}\right) .  \label{Eq_5}
\end{equation}%
By the notation of Lemma \ref{Lem_h_p(x)} we have 
\begin{eqnarray*}
h_{p}\left( q\right) &=&\frac{1}{p-1}-\frac{\left( \log _{p}q\right) }{q-1}
\\
1-\frac{\left( \log _{p}q\right) }{q-1} &=&1-\frac{1}{p-1}+\frac{1}{p-1}-%
\frac{\left( \log _{p}q\right) }{q-1} \\
&=&\frac{p-2}{p-1}+h_{p}\left( q\right) .
\end{eqnarray*}%
Applying this to (\ref{Eq_5}) we get that (\ref{Eq_5}) is equivalent to%
\begin{equation}
\left\lfloor \log _{p}n\right\rfloor <\left( n-1\right) h_{p}\left( q\right)
-\frac{p-2}{p-1}.  \label{Eq_6}
\end{equation}%
By Lemma \ref{Lem_h_p(x)}, $h_{p}\left( q\right) \geq h_{p}\left( p_{\text{%
succ}}\right) >0$, and hence, any positive integer $n$ satisfying 
\begin{equation}
\left\lfloor \log _{p}n\right\rfloor <\left( n-1\right) h_{p}\left( p_{\text{%
succ}}\right) -\frac{p-2}{p-1}\text{,}  \label{Eq_7}
\end{equation}%
also satisfies (\ref{Eq_6}) for all primes $q$ such that $p<q\leq n$. Since $%
\left\lfloor \log _{p}n\right\rfloor \leq \log _{p}n$, any positive integer $%
n$ satisfying%
\begin{equation}
\log _{p}n<\left( n-1\right) h_{p}\left( p_{\text{succ}}\right) -\frac{p-2}{%
p-1}\text{,}  \label{Eq_(7.1)}
\end{equation}%
also satisfies (\ref{Eq_7}). Now (\ref{Eq_(7.1)}) is a special case of the
following inequality:%
\begin{equation}
\log _{p}n<n\cdot a-b.  \label{Eq_(7.2)}
\end{equation}%
where $a=h_{p}\left( p_{\text{succ}}\right) $ and $b=a+\frac{p-2}{p-1}$.
Note that $a$ and $b$ are functions of $p$, which are independent of $n$,
and that $a$ is a positive real number by Lemma (\ref{Lem_h_p(x)}). Since
the right hand side of (\ref{Eq_(7.2)}) is a linear function of $n$ with a
positive slope, it eventually dominates $\log _{p}n$ on the left-hand side
of (\ref{Eq_(7.2)}), which means that there exists $n_{0}\in \mathbb{N}$
such that Inequality \eqref{Eq_(7.2)} holds true for all $n\geq n_{0}$.
\end{proof}

\begin{proof}[Proof of Corollary~\ref{Coro_2_domination}]
By the proof of Theorem \ref{Th_p<q} (see (\ref{Eq_7})) there exists $%
n_{0}\in \mathbb{N}$ such that 
\begin{equation}\label{ast1}
 ~\left\lfloor \log _{2}n\right\rfloor <\left( n-1\right)
h_{2}\left( 3\right) \text{; }\forall n\geq n_{0}\text{,}
\end{equation}
and for this $n_{0}$, $2^{\nu _{2}\left( n!\right) }>q^{\nu _{q}\left(
n!\right) }$ for all primes $q>2$ and all integers $n\geq n_{0}$. By (\ref%
{Def_h_p(x)}) we have 
\begin{equation*}
h_{2}\left( 2_{\text{succ}}\right) =h_{2}\left( 3\right) =1-\frac{1}{2}\frac{%
\ln 3}{\ln 2}>0.2075>\frac{1}{5}  \label{Eq_h_2(3)}
\end{equation*}
and hence we get that any $n$ satisfying 
\begin{equation}\label{ast2}
 \left\lfloor \log _{2}n\right\rfloor \leq \frac{n-1%
}{5}\text{,}
\end{equation}%
also satisfies \eqref{ast1}. By inspection, the smallest integer
solution of \eqref{ast2} is $n=21$, and \eqref{ast2} also holds true for all $21\leq n\leq 31$. Now, $\left\lfloor \log
_{2}n\right\rfloor $ is a step function which is constant on each step (an
interval of the form $\left[ 2^{k},2^{k+1}-1\right] $) and increases by one
unit from step to step. On the other hand, $\frac{n-1}{5}$ is a linear
function of $n$ that increases by $\frac{2^{k}}{5}$ along the step $\left[
2^{k},2^{k+1}-1\right] $. Since $n=32$ corresponds to $k=5$, and $\frac{2^{5}%
}{5}>6>1$, it is clear that \eqref{ast2} holds also for all $%
n\geq 32$. We have thus shown that $2^{\nu _{2}\left( n!\right) }>q^{\nu
_{q}\left( n!\right) }$ holds true for all $n\geq 21$ and all primes $q>2$.
It remains to consider the interval $2\leq n\leq 20$. Since this interval is
finite and the relevant values of $q$ are the odd primes in this same
interval, namely the primes $3,5,7,11,13,17,19$, one can check
\textquotedblleft by hand\textquotedblright (e.g., it is easy to write a
GAP \cite{GAP2025} program) that the inequality $2^{\nu _{2}\left( n!\right)
}>q^{\nu _{q}\left( n!\right) }$ is true for all integers $2\leq n\leq 20$
and all primes $q\in \left\{ 3,5,7,11,13,17,19\right\} $ with the single
exception $q=n=3$ -- see Section~\ref{Sect_GAP_Code}.
\end{proof}

\subsection{Proof of Theorem~\ref{Th_twin_primes}} \label{Subsect_twin_primes_proof}

Let $n\geq 1$ and $b\geq 2$ be integers. Recall that the base $b$
representation of $n$ is the unique sequence 
\begin{equation*}
\left( n\right) _{b}:=\left( d_{l}\text{,}d_{l-1},\ldots ,d_{1},d_{0}\right)
\end{equation*}
of non-negative integers (the base $b$ digits of $n$) $d_{i}$, $0\leq i\leq
l $ such that $d_{i}\in \left\{0,\ldots ,b-1\right\} $, $d_{l}\geq 1$,
and 
\begin{equation*}
n=d_{l}b^{l}+d_{l-1}b^{l-1}+\cdots +d_{1}b+d_{0}.
\end{equation*}%
Furthermore, we say that $n$ is an $(l+1)$-digit number in base $b$, and $%
s_{b}\left( n\right) :=\sum\limits_{i=0}^{l}d_{i}$.

\begin{lemma}
\label{Lem_basic_n-s_b(n)_properties}With notation as above, $n-s_{b}\left(
n\right) =0$ for all $1\leq n\leq b-1$, and $n-s_{b}\left( n\right) >0$ for
all $n\geq b$. Furthermore, $\left( b-1\right) |\left( n-s_{b}\left(
n\right) \right) $.
\end{lemma}

\begin{proof}
The first claim is clear. For $n\geq b$ we have: 
\begin{equation*}
n-s_{b}\left( n\right)
=\sum\limits_{i=0}^{l}d_{i}b^{i}-\sum\limits_{i=0}^{l}d_{i}=\sum%
\limits_{i=0}^{l}d_{i}\left( b^{i}-1\right)
=\sum\limits_{i=1}^{l}d_{i}\left( b^{i}-1\right) \geq b^{l}-1>0\text{,}
\end{equation*}%
where $d_{l}\geq 1$ was used. Moreover, $\left( b-1\right) |\left(
b^{i}-1\right) $ for all $1\leq i\leq l$, so $\left( b-1\right) |\left(
n-s_{b}\left( n\right) \right) $ follows.
\end{proof}

Let $p$ be a prime such that $p+2$ is also a prime. Set 
\begin{equation*}
q:=p+2\text{, }k:=\frac{p^{2}+p}{2}-1.
\end{equation*}%
In order to prove Theorem \ref{Th_twin_primes}, we have to prove that 
\begin{equation*}
q^{\nu _{q}\left( k!\right) }>p^{\nu _{p}\left( k!\right) }\text{;}
\end{equation*}%
and

\begin{equation}
~~~~~~p^{\nu _{p}\left( n!\right) }>q^{\nu _{q}\left( n!\right) }\text{; }%
\forall n\geq k+1.  \label{Eq_9}
\end{equation}

\begin{lemma}
\label{Lem_nu_q_(k!)=nu_p_(k!)} The identity
\begin{equation*}
\nu _{q}\left( k!\right) =\nu _{p}\left( k!\right) =\frac{p-1}{2}
\end{equation*}%
holds, and consequently, $q^{\nu _{q}\left( k!\right) }>p^{\nu _{p}\left( k!\right)
} $.
\end{lemma}

\begin{proof}
We calculate the exponents $\nu _{q}\left( k!\right) $ and $\nu _{p}\left(
k!\right) $ using (\ref{Legendre_formula_alternative}). We have%
\begin{equation*}
k=\left( \frac{p-1}{2}\right) \left( p+2\right) \Longrightarrow \left(
k\right) _{p+2}=\left( \frac{p-1}{2},0\right) \Longrightarrow s_{p+2}\left(
k\right) =\frac{p-1}{2}.
\end{equation*}%
Similarly%
\begin{equation*}
k=\left( \frac{p-1}{2}\right) p+p-1\Longrightarrow \left( k\right)
_{p}=\left( \frac{p-1}{2},p-1\right) \Longrightarrow s_{p}\left( k\right) =%
\frac{3}{2}\left( p-1\right) .
\end{equation*}%
Substituting in (\ref{Legendre_formula_alternative}) gives:%
\begin{equation*}
\nu _{q}\left( k!\right) =\frac{k-s_{q}\left( k\right) }{q-1}=\frac{\left( 
\frac{p-1}{2}\right) \left( p+2\right) -\frac{p-1}{2}}{p+1}=\frac{p-1}{2}
\end{equation*}

\begin{equation*}
\nu _{p}\left( k!\right) =\frac{k-s_{p}\left( k\right) }{p-1}=\frac{\left( 
\frac{p-1}{2}\right) p+p-1-\frac{3}{2}\left( p-1\right) }{p-1}=\frac{p}{2}+1-%
\frac{3}{2}=\frac{p-1}{2}.
\end{equation*}%
Having proved $\nu _{q}\left( k!\right) =\nu _{p}\left( k!\right) $, $q^{\nu
_{q}\left( k!\right) }>p^{\nu _{p}\left( k!\right) }$ follows from $q>p$.
\end{proof}

Proving (\ref{Eq_9}) requires more effort. First, (\ref{Eq_9}) is equivalent
to 
\begin{equation}
r\left( n,p\right) :=\frac{n-s_{p}\left( n\right) }{n-s_{p+2}\left( n\right) 
}>\left( \frac{\ln \left( p+2\right) }{\ln p}\right) \left( \frac{p-1}{p+1}%
\right) \text{; }\forall n\geq k+1=\frac{p^{2}+p}{2}\text{,}  \label{Eq_10}
\end{equation}%
by Lemma \ref{Lem_from_powers_to_logs}. Note that $n-s_{p+2}\left( n\right)
>0$ by Lemma \ref{Lem_basic_n-s_b(n)_properties} ($p+2\leq n$ since $n\geq
k+1$ and $p\geq 3$).

For a fixed odd prime $p$ we view $r\left( n,p\right) $ as a function from
the set of all integers $n$ satisfying $n\geq k+1$ into the set of positive
real numbers.

\begin{lemma} We obtain the value
\label{Lem_r(k+1,p)=1}$r\left( k+1,p\right) =1$.
\end{lemma}

\begin{proof}
Using $\left( k\right) _{p}$ and $\left( k\right) _{p+2}$ from the proof of
Lemma \ref{Lem_nu_q_(k!)=nu_p_(k!)}, gives 
\begin{equation*}
\left( k+1\right) _{p}=\left( \frac{p+1}{2},0\right) \text{ and }\left(
k+1\right) _{p+2}=\left( \frac{p-1}{2},1\right) .
\end{equation*}%
Hence $s_{p}\left( k+1\right) =s_{p+2}\left( k+1\right) =\frac{p+1}{2}$, and
the claim follows.
\end{proof}

\begin{lemma}
\label{Lem_r(n,p)_has_a global_minimum}The function $r\left( n,p\right) $ has a global
minimum in the interval $n\geq \frac{p^{2}+p}{2}$.
\end{lemma}

\begin{proof}
First we prove $\lim_{n\rightarrow \infty }r\left( n,p\right) =1$. Write 
\begin{equation*}
r\left( n,p\right) :=\frac{n-s_{p}\left( n\right) }{n-s_{p+2}\left( n\right) 
}=\frac{1-\frac{s_{p}\left( n\right) }{n}}{1-\frac{s_{p+2}\left( n\right) }{n%
}}\text{,}
\end{equation*}%
and use the bounds $1\leq s_{p}\left( n\right) \leq \left( p-1\right) \left(
1+\log _{p}n\right) $ (see the proof of Theorem \ref{Th_p<q}). Since a
similar inequality applies to $s_{p+2}\left( n\right) $, we get: 
\begin{equation}
\frac{1-\frac{\left( p-1\right) \left( 1+\log _{p}n\right) }{n}}{1-\frac{1}{n%
}}\leq r\left( n,p\right) \leq \frac{1-\frac{1}{n}}{1-\frac{\left(
p+1\right) \left( 1+\log _{p+2}n\right) }{n}}.  \label{Eq_14}
\end{equation}%
Since $p$ is fixed, we have 
\begin{equation*}
\lim_{n\rightarrow \infty }\frac{\left( p-1\right) \left( 1+\log
_{p}n\right) }{n}=\lim_{n\rightarrow \infty }\frac{\left( p+1\right) \left(
1+\log _{p+2}n\right) }{n}=0\text{,}
\end{equation*}%
and $\lim_{n\rightarrow \infty }r\left( n,p\right) =1$ follows. Now, assume
by contradiction that $r\left( n,p\right) $ has no global minimum in the
interval $n\geq k+1$. By Lemma \ref{Lem_r(k+1,p)=1} we have $r\left(
k+1,p\right) =1$. Since this cannot be a global minimum in the interval $%
n\geq k+1$, there exists an infinite monotonically increasing sequence of
integers $k+1=i_{0}<i_{1}<i_{2}<\cdots $ such that the sequence 
\begin{equation*}
r\left( i_{0},p\right) >r\left( i_{1},p\right) >r\left( i_{2},p\right)
>\cdots
\end{equation*}
is a monotonically decreasing sequence of positive real numbers. It follows
that this sequence is bounded from below by some nonnegative real number
less than $1$. Therefore, there exists $0\leq l<1$ such that $$\lim_{j\rightarrow \infty }r\left( i_{j},p\right) =l.$$ This contradicts $
\lim_{n\rightarrow \infty }r\left( n,p\right) =1$.
\end{proof}

We will prove that (\ref{Eq_10}) holds true for all $n\geq k+1$ by showing
that the global minimum of $r\left( n,p\right) $ in the interval $n\geq k+1$
satisfies (\ref{Eq_10}). The following theorem gives the necessary
information about the global minimum of $r\left( n,p\right) $.

\begin{theorem}
\label{Th_r(n,p)_global_minimum}
\begin{enumerate}
    \item[(a)] The smallest integer $m$ at which the global minimum of $r\left(
n,p\right) $ in the interval $n\geq k+1$ occurs is given by%
\begin{equation}
m=\left\{ 
\begin{array}{c}
\left( p+2\right) ^{2}\text{ if }p\in \left\{ 3,5\right\} \\ 
p^{2}-4\text{ if }p\geq 11.%
\end{array}%
\right.  \label{Eq_11}
\end{equation}
\item[(b)] The value of the global minimum of $r\left( n,p\right) $ in the interval 
$n\geq k+1$ is given by:%
\begin{eqnarray*}
r\left( m,3\right) &=&r\left( 25,3\right) =\frac{5}{6} \\
r\left( m,5\right) &=&r\left( 49,5\right) =\frac{5}{6} \\
r\left( m,p\right) &=&r\left( p^{2}-4,p\right) =\frac{\left( p-1\right) ^{2}%
}{\left( p-2\right) \left( p+1\right) }\text{; }\forall p\geq 11.
\end{eqnarray*}
\end{enumerate} 

\end{theorem}

Before proving Theorem \ref{Th_r(n,p)_global_minimum}, we check that it
implies (\ref{Eq_9}). As we already saw, (\ref{Eq_9}) is equivalent to (\ref%
{Eq_10}). Since $r\left( m,p\right) $ is the global minimum of $r\left(
n,p\right) $ in the interval $n\geq k+1$, (\ref{Eq_10}) follows if and only if

\begin{equation}
r\left( m,p\right) >\left( \frac{\ln \left( p+2\right) }{\ln p}\right)
\left( \frac{p-1}{p+1}\right) .  \label{Eq_13}
\end{equation}

\begin{lemma}
Assuming that Theorem \ref{Th_r(n,p)_global_minimum} is correct, (\ref{Eq_13}%
) holds true for all $p$ such that $p$ and $p+2$ are twin primes.
\end{lemma}

\begin{proof}
For each $p$ we compare the value of the right hand side of (\ref{Eq_13})
with the value of $r\left( m,p\right) $ as given by Theorem \ref%
{Th_r(n,p)_global_minimum}.

1. $p=3$ 
\begin{equation*}
\frac{\ln \left( p+2\right) }{\ln p}\cdot \frac{p-1}{p+1}=\frac{\ln 5}{\ln 3}%
\cdot \frac{2}{4}<0.733<r\left( 25,3\right) =\frac{5}{6}=0.8333....
\end{equation*}

2. $p=5$ 
\begin{equation*}
\frac{\ln \left( p+2\right) }{\ln p}\cdot \frac{p-1}{p+1}=\frac{\ln 7}{\ln 5}%
\cdot \frac{4}{6}<0.807<r\left( 49,5\right) =\frac{5}{6}=0.8333...
\end{equation*}

3. $p\geq 11$. We have the following chain of equivalences:
\begin{align*}
&\frac{\ln \left( p+2\right) }{\ln p}\cdot \frac{p-1}{p+1}<r\left( m,p\right)
=\frac{\left( p-1\right) ^{2}}{\left( p-2\right) \left( p+1\right) } \\
&\Longleftrightarrow
\frac{\ln \left( p+2\right) }{\ln p}<\frac{p-1}{p-2} \\
&\Longleftrightarrow \left( p-1\right) \ln p>\left( p-2\right) \ln \left( p+2\right) \\
&\Longleftrightarrow 
\left( p-1\right) \ln p-\left( p-2\right) \ln \left( p+2\right) >0.
\end{align*}
Set%
\begin{equation*}
f\left( x\right) :=\left( x-1\right) \ln x-\left( x-2\right) \ln \left(
x+2\right) \text{; }\forall x\geq 11.
\end{equation*}%
We will show that $f\left( x\right) $ is positive in its domain of
definition.%
\begin{align*}
f\left( x\right) & =x\left( \ln x-\ln \left( x+2\right) \right) +2\ln \left(
x+2\right) -\ln x \\
& =x\ln \left( \frac{x}{x+2}\right) +\ln \left( x+2\right) ^{2}-\ln x \\
& =-x\ln \left( \frac{x+2}{x}\right) +\ln \left( \frac{\left( x+2\right) ^{2}%
}{x}\right) \\
& =-\ln \left( \left( \frac{x+2}{x}\right) ^{x}\right) +\ln \left( \frac{%
\left( x+2\right) ^{2}}{x}\right) \\
& =-\ln \left( \left( 1+\frac{2}{x}\right) ^{x}\right) +\ln \left( \frac{%
\left( x+2\right) ^{2}}{x}\right) .
\end{align*}%
Using 
\begin{equation*}
\lim_{x\rightarrow \infty }\left( 1+\frac{2}{x}\right) ^{x}=e^{2}\text{,}
\end{equation*}%
we have 
\begin{equation*}
-\lim_{x\rightarrow \infty }\ln \left( \left( 1+\frac{2}{x}\right)
^{x}\right) =-\ln e^{2}=-2.
\end{equation*}%
Moreover, $-\ln \left( \left( 1+\frac{2}{x}\right) ^{x}\right) $ is
monotonically decreasing towards its limit. On the other hand, $\ln \left( 
\frac{\left( x+2\right) ^{2}}{x}\right) $ is monotonically increasing for $x\geq 11$, and for $x=11$, we have 
\begin{equation*}
\ln \left( \frac{\left( 11+2\right) ^{2}}{11}\right) =\ln \frac{13}{11}+\ln
13>\ln \frac{13}{11}+\ln e^{2}>2.
\end{equation*}%
We get:%
\begin{equation*}
f\left( x\right) >-\lim_{x\rightarrow \infty }\ln \left( \left( 1+\frac{2}{x}%
\right) ^{x}\right) +\ln \left( \frac{\left( 11+2\right) ^{2}}{11}\right)
>-2+2=0.\qedhere
\end{equation*}
\end{proof}

Now we prove Theorem \ref{Th_r(n,p)_global_minimum}. Given that part (a) is
established, deriving part (b) from part (a) is an easy task.

\begin{proof}[Proof of Theorem \protect\ref{Th_r(n,p)_global_minimum}(b)
assuming (a)]
These are just routine calculations that are required for substituting $%
n\leftarrow m$ into $r\left( n,p\right) :=\frac{n-s_{p}\left( n\right) }{%
n-s_{p+2}\left( n\right) }$.\ Equality (\ref{Eq_11}) immediately yields:

1. If $p=3$ then $m=25$ and 
\begin{gather*}
\left( m\right) _{3}=\left( 25\right) _{3}=\left( 2,2,1\right) \text{, }%
s_{3}\left( 25\right) =5 \\
\left( m\right) _{3+2}=\left( 25\right) _{5}=\left( 1,0,0\right) \text{, }%
s_{5}\left( 25\right) =1.
\end{gather*}

2. If $p=5$ then $m=49$ and%
\begin{gather*}
\left( m\right) _{5}=\left( 49\right) _{5}=\left( 1,4,4\right) \text{, }%
s_{5}\left( 49\right) =9 \\
\left( m\right) _{5+2}=\left( 49\right) _{7}=\left( 1,0,0\right) \text{, }%
s_{7}\left( 49\right) =1.
\end{gather*}

3. If $p\geq 11$ then $m=p^{2}-4=\left( p-2\right) \left( p+2\right) =\left(
p-1\right) p+p-4$ and hence 
\begin{gather*}
\left( m\right) _{p}=\left( p-1,p-4\right) \text{, }s_{p}\left(
p^{2}-4\right) =2p-5 \\
\left( m\right) _{p+2}=\left( p-2,0\right) \text{, }s_{p+2}\left(
p^{2}-4\right) =p-2
\end{gather*}

Recall that%
\begin{equation*}
r\left( m,p\right) =\frac{m-s_{p}\left( m\right) }{m-s_{p+2}\left( m\right) }%
.
\end{equation*}

The explicit expressions for $r\left( m,p\right) $ in the various cases are
as follows:%
\begin{equation*}
r\left( m,3\right) =\frac{25-s_{3}\left( 25\right) }{25-s_{5}\left(
25\right) }=\frac{25-5}{25-1}=\frac{5}{6}.
\end{equation*}%
\begin{equation*}
r\left( m,5\right) =\frac{49-s_{5}\left( 49\right) }{49-s_{7}\left(
49\right) }=\frac{49-9}{49-1}=\frac{5}{6}.
\end{equation*}

For all $p\geq 11$ we get%
\begin{gather*}
r\left( m,p\right) =\frac{p^{2}-4-s_{p}\left( p^{2}-4\right) }{%
p^{2}-4-s_{p+2}\left( p^{2}-4\right) } \\
=\frac{p^{2}-4-\left( 2p-5\right) }{p^{2}-4-\left( p-2\right) }=\frac{\left(
p-1\right) ^{2}}{\left( p-2\right) \left( p+1\right) }.
\end{gather*}
\end{proof}

Finally, we turn to prove part (a) of Theorem \ref{Th_r(n,p)_global_minimum}.
First, we show that the global minimum of $r\left( n,p\right) $ in the
interval $n\geq k+1$ occurs at $m$, and then we show that for all $k+1\leq
n<m$ we have $r\left( n,p\right) >r\left( m,p\right) $. Therefore, $m$ is the
smallest integer in the interval $n\geq k+1$ at which the global minimum of $%
r\left( n,p\right) $ in this interval occurs.

We begin by arguing that the global minimum of $r\left( n,p\right) $ occurs
\textquotedblleft not too far\textquotedblright\ from $k+1$. To this end, we
use the lower bound (\ref{Eq_14}) on $r\left( n,p\right) $:

\begin{equation*}
r\left( n,p\right) \geq \frac{1-\frac{\left( p-1\right) \left( \log
_{p}\left( n\right) +1\right) }{n}}{1-\frac{1}{n}}>1-\frac{\left( p-1\right)
\left( \log _{p}\left( n\right) +1\right) }{n}\text{; }\forall n\geq q\text{.%
}
\end{equation*}%
Since $\log _{p}\left( n\right) =\frac{\ln n}{\ln p}$, we have 
\begin{equation}
~r\left( n,p\right) >1-\left( p-1\right) \left( \frac{\ln n}{\ln p}+1\right) 
\frac{1}{n}\text{; }\forall n\geq q.  \label{Eq_15}
\end{equation}

\begin{lemma}
\label{Lem_f(x)}The real function 
\begin{equation*}
f\left( x\right) :=\left( \frac{\ln x}{\ln p}+1\right) \frac{1}{x}\text{; }%
\forall x\geq 1
\end{equation*}%
is positive and monotonically decreasing.
\end{lemma}

\begin{proof}
Positivity is clear. To show that the function is monotonically decreasing,
we calculate its first derivative:%
\begin{equation*}
f^{\prime }\left( x\right) =\frac{1}{\ln p}\frac{1}{x^{2}}+\left( \frac{\ln x%
}{\ln p}+1\right) \frac{-1}{x^{2}}=\frac{1}{x^{2}}\left( \frac{1}{\ln p}-%
\frac{\ln x}{\ln p}-1\right) .
\end{equation*}%
Since $p\geq 3>e$, we have $\frac{1}{\ln p}<1$ and $\frac{1}{\ln p}-1<0$.
Hence, since $\frac{\ln x}{\ln p}\geq 0$ for all $x\geq 1$, $f^{\prime
}\left( x\right) <0$, finishing the proof of the lemma.
\end{proof}

\begin{corollary}
\label{Coro_n>=ntilde}Let $\widetilde{n}\geq q$ be an integer. Then, for all
integers $n\geq \widetilde{n}$ we have 
\begin{equation*}
r\left( n,p\right) >1-\left( p-1\right) \left( \frac{\ln \widetilde{n}}{\ln p%
}+1\right) \frac{1}{\widetilde{n}}.
\end{equation*}
\end{corollary}

\begin{proof}
Let $n\geq \widetilde{n}$ be an integer. Since $\widetilde{n}\geq q$ and $n\geq \widetilde{n}$ we have $n\geq q$ and therefore, by (\ref{Eq_15}), 
\begin{equation*}
r\left( n,p\right) >1-\left( p-1\right) \left( \frac{\ln n}{\ln p}+1\right) 
\frac{1}{n}.
\end{equation*}%
\ \ In the notation of Lemma \ref{Lem_f(x)} we have 
\begin{equation*}
r\left( n,p\right) >1-\left( p-1\right) f\left( n\right) .
\end{equation*}%
Since $f\left( x\right) $ is monotonically decreasing for all $x\geq 1$, $%
1-\left( p-1\right) f\left( x\right) $ is monotonically increasing for all $%
x\geq 1$. Hence, since $n\geq \widetilde{n}$,%
\begin{equation*}
r\left( n,p\right) >1-\left( p-1\right) f\left( n\right) \geq 1-\left(
p-1\right) f\left( \widetilde{n}\right) .
\end{equation*}
\end{proof}

To utilize the last corollary, we look for $\widetilde{n}\geq q$ such that 
\begin{equation}
1-\left( p-1\right) \left( \frac{\ln \widetilde{n}}{\ln p}+1\right) \frac{1}{%
\widetilde{n}}\geq r\left( m,p\right) \text{,}  \label{Eq_16}
\end{equation}%
where $m$ is given by (\ref{Eq_11}). Given such $\widetilde{n}$, Corollary %
\ref{Coro_n>=ntilde} implies $r\left( n,p\right) >r\left( m,p\right) $ for
all $n\geq \widetilde{n}$, and it will remain to establish the claim of part
(a) of Theorem \ref{Th_r(n,p)_global_minimum} for the range $k+1\leq n<%
\widetilde{n}$.

\begin{lemma}
Define $m$ by (\ref{Eq_11}).

(a) For $p=3$, (\ref{Eq_16}) holds true for $\widetilde{n}=3^{4}$.

(b) For $p=5$, (\ref{Eq_16}) holds true for $\widetilde{n}=5^{3}$.

(c) For any prime $p\geq 11$, (\ref{Eq_16}) holds true for $\widetilde{n}%
=4p^{2}$.
\end{lemma}

\begin{proof}
(a) If $p=3$ and $\widetilde{n}=3^{4}$, we have:%
\[
1-\left( p-1\right) \left( \frac{\ln \widetilde{n}}{\ln p}+1\right) \frac{1}{%
\widetilde{n}}=1-2\left( \frac{\ln 3^{4}}{\ln 3}+1\right) \frac{1}{81}=1-%
\frac{10}{81}
=\frac{71}{81}>\frac{5}{6}=r\left( 25,3\right).
\]

(b) If $p=5$ and $\widetilde{n}=5^{3}$, we have:%
\begin{gather*}
1-\left( p-1\right) \left( \frac{\ln \widetilde{n}}{\ln p}+1\right) \frac{1}{%
\widetilde{n}}=1-4\left( \frac{\ln 5^{3}}{\ln 5}+1\right) \frac{1}{5^{3}}
=1-\frac{4\cdot 4}{125}=\frac{109}{125}>\frac{5}{6}=r\left( 49,5\right).
\end{gather*}

(c) If $p\geq 11$ and $\widetilde{n}=4p^{2}$, we have to show that 
\begin{equation*}
1-\left( p-1\right) \left( \frac{\ln 4p^{2}}{\ln p}+1\right) \frac{1}{4p^{2}}%
\geq r\left( m,p\right) \text{; }\forall p\geq 11.
\end{equation*}%
We have 
\begin{equation*}
\left( \frac{\ln 4p^{2}}{\ln p}+1\right) \frac{1}{4p^{2}}=\left( \frac{\ln
4+\ln p^{2}}{\ln p}+1\right) \frac{1}{4p^{2}}=\frac{1}{4}\left( 3+\frac{\ln 4%
}{\ln p}\right) \frac{1}{p^{2}}.
\end{equation*}%
Therefore, setting $c_{p}:=\frac{1}{4}\left( 3+\frac{\ln 4}{\ln p}\right) $,
we have to show that 
\begin{equation*}
1-\frac{c_{p}\left( p-1\right) }{p^{2}}\geq \frac{\left( p-1\right) ^{2}}{%
\left( p-2\right) \left( p+1\right) }\text{; }\forall p\geq 11.
\end{equation*}%
We have:%
\begin{align*}
&1-\frac{c_{p}\left( p-1\right) }{p^{2}}\geq \frac{\left( p-1\right) ^{2}}{%
\left( p-2\right) \left( p+1\right) } \\
&\Longleftrightarrow
p^{2}\left( p-2\right) \left( p+1\right) -p^{2}\left( p-1\right) ^{2}\geq
c_{p}\left( p-1\right) \left( p-2\right) \left( p+1\right) \\
&\Longleftrightarrow
p^{2}\left( p^{2}-p-2\right) -p^{2}\left( p^{2}-2p+1\right) \geq c_{p}\left(
p-2\right) \left( p^{2}-1\right) \\
&\Longleftrightarrow 
p^{3}-3p^{2}\geq c_{p}\left( p^{3}-2p^{2}-p+2\right).
\end{align*}
Set 
\begin{gather*}
f\left( x\right) :=x^{3}-3x^{2}-c_{x}\left( x^{3}-2x^{2}-x+2\right) \\
\text{where }c_{x}:=\frac{1}{4}\left( 3+\frac{\ln 4}{\ln x}\right)
\end{gather*}%
Then 
\begin{equation*}
f^{\prime }\left( x\right) =3x^{2}-6x-\left( c_{x}\right) ^{\prime }\left(
x^{3}-2x^{2}-x+2\right) -c_{x}\left( 3x^{2}-4x-1\right) .
\end{equation*}%
Since%
\begin{equation*}
\left( c_{x}\right) ^{\prime }:=\frac{1}{4}\left( 3+\frac{\ln 4}{\ln x}%
\right) ^{\prime }=\frac{\ln 4}{4}\frac{-1}{x\left( \ln x\right) ^{2}}\text{,%
}
\end{equation*}%
we have%
\begin{eqnarray*}
f^{\prime }\left( x\right) &=&3x^{2}-6x+\frac{\ln 4}{4}\frac{1}{x\left( \ln
x\right) ^{2}}\left( x^{3}-2x^{2}-x+2\right) -c_{x}\left( 3x^{2}-4x-1\right)
\\
&=&\left( 3x^{2}-4x-1\right) \left( 1-c_{x}\right) +\frac{\ln 4}{4}\frac{1}{%
x\left( \ln x\right) ^{2}}\left( x-2\right) \left( x^{2}-1\right) +1-2x\text{%
.}
\end{eqnarray*}%
In order to prove that $f^{\prime }\left( x\right) >0$ for all $x\geq 11$,
start with 
\begin{equation*}
\frac{\ln 4}{4}\frac{1}{x\left( \ln x\right) ^{2}}\left( x-2\right) \left(
x^{2}-1\right) +1>1\text{; }\forall x\geq 11.
\end{equation*}%
Then observe that $\left( c_{x}\right) ^{\prime }<0$ for all $x\geq 11$, so $%
c_{x}$ is maximal in $[11,\infty )$ at $x=11$ and hence $c_{x}\leq
c_{11}<0.895<\frac{9}{10}$, which gives $1-c_{x}>\frac{1}{10}$ for all $x\geq
11$. Since $3x^{2}-4x-1>0$ for all $x\geq 11$, we have 
\begin{equation*}
\left( 3x^{2}-4x-1\right) \left( 1-c_{x}\right) -2x>\frac{1}{10}\left(
3x^{2}-4x-1\right) -2x\text{; }\forall x\geq 11\text{,}
\end{equation*}%
and hence it suffices to prove that $$\frac{1}{10}\left( 3x^{2}-4x-1\right)
-2x>0; \forall x\geq 11.$$ But the last inequality is equivalent to $%
3x^{2}-24x=3x\left( x-8\right) >1$, which is true for all $x\geq 11$. This
concludes the proof that $f^{\prime }\left( x\right) >0$ for all $x\geq 11$.
Using $c_{x}<0.895$ for all $x\geq 11$, we check that 
\begin{equation*}
f\left( 11\right) >968-1080\cdot 0.895=1.4>0\text{,}
\end{equation*}
which, together with the positivity of $f^{\prime }\left( x\right) $ for all $%
x\geq 11$, concludes the proof of the lemma.
\end{proof}

\begin{proposition}
\label{Prop_3_5_ntilde}Let $p\in \left\{ 3,5\right\} $, $k:=\frac{p^{2}+p}{2}%
-1$, $m:=\left( p+2\right) ^{2}$ and $\widetilde{n}:=p^{4}$ if $p=3$ and $%
\widetilde{n}:=p^{3}$ if $p=5$. Then $r\left( n,p\right) \geq r\left(
m,p\right) =\frac{5}{6}$ for all $k+1\leq n\leq \widetilde{n}$, and if $%
k+1\leq n<m$ then $r\left( n,p\right) >r\left( m,p\right) $.
\end{proposition}

\begin{proof}
GAP \cite{GAP2025} calculation.
\end{proof}

\begin{proposition}
Let $p\geq 11$ be a prime such that $q:=p+2$ is a prime. Let $k:=\frac{%
p^{2}+p}{2}-1$, $m:=p^{2}-4$ and $\widetilde{n}:=4p^{2}$. Then $$r\left(
n,p\right) \geq r\left( m,p\right) =\frac{\left( p-1\right) ^{2}}{\left(
p-2\right) \left( p+1\right) }$$ for all $k+1\leq n\leq \widetilde{n}$, and
if $k+1\leq n<m$ then $r\left( n,p\right) >r\left( m,p\right) $.
\end{proposition}

\begin{proof}
Let $k+1\leq n\leq \widetilde{n}$, and let 
\begin{equation*}
\left( n\right) _{p}:=\left( d_{l},d_{l-1},\ldots ,d_{1},d_{0}\right) \text{
and }\left( n\right) _{q}:=\left( e_{l^{\prime }},e_{l^{\prime }-1},\ldots
,e_{1},e_{0}\right)
\end{equation*}%
be, respectively, the base $p$ and the base $q$ digits of $n$. Then (see
Lemma \ref{Lem_basic_n-s_b(n)_properties}) 
\begin{equation*}
n-s_{p}\left( n\right)
=\sum\limits_{i=0}^{l}d_{i}p^{i}-\sum\limits_{i=0}^{l}d_{i}=\sum%
\limits_{i=0}^{l}d_{i}\left( p^{i}-1\right)
=\sum\limits_{i=1}^{l}d_{i}\left( p^{i}-1\right) \text{,}
\end{equation*}%
and, similarly, 
\begin{equation*}
n-s_{q}\left( n\right) =\sum\limits_{i=1}^{l^{\prime }}e_{i}\left( \left(
p+2\right) ^{i}-1\right) .
\end{equation*}%
Since $k+1\leq n<\widetilde{n}=4p^{2}$, and $p\geq 11$, we have $%
4p^{2}<4\left( p+2\right) ^{2}<p^{3}$ and hence $l,l^{\prime }\leq 2$. On
the other hand, $\frac{p^{2}+p}{2}\leq n$ implies $l,l^{\prime }\geq 1$.
Therefore,
$$\begin{aligned}
   r\left( n,p\right) :&=\frac{n-s_{p}\left( n\right) }{n-s_{p+2}\left(
n\right) }=\frac{d_{2}\cdot \left( p^{2}-1\right) +d_{1}\cdot \left(
p-1\right) }{e_{2}\cdot \left( \left( p+2\right) ^{2}-1\right) +e_{1}\cdot
\left( p+1\right) } \\
&=\frac{p-1}{p+1}\cdot \frac{d_{2}\cdot \left( p+1\right) +d_{1}}{%
e_{2}\cdot \left( p+3\right) +e_{1}}.
\end{aligned}$$
Hence $r\left( n,p\right) \geq r\left( m,p\right) =\frac{\left( p-1\right)
^{2}}{\left( p-2\right) \left( p+1\right) }$ for all $p\geq 11$ if and only if%
\begin{equation}\label{ast}
 ~\frac{d_{2}\cdot \left( p+1\right) +d_{1}}{e_{2}\cdot
\left( p+3\right) +e_{1}}\geq \frac{p-1}{p-2}, \ \forall p\geq 11\text{.%
}
\end{equation}
Our aim is to show that Condition \eqref{ast} is true.
\begin{enumerate}
    \item[(1).] Suppose that $n$ is a $2$-digit number in base $p$, or in other words, $%
d_{2}=0$. This implies that also $e_{2}=0$ and hence $e_{1}>0$ and condition 
\eqref{ast} reads $\frac{d_{1}}{e_{1}}\geq \frac{p-1}{p-2}$.
Since $e_{2}=d_{2}=0$, we get: 
\begin{equation*}
n=d_{1}p+d_{0}=e_{1}\left( p+2\right) +e_{0}=e_{1}p+2e_{1}+e_{0}.
\end{equation*}%
Since $d_{0}<p$, we must have $e_{1}\leq d_{1}\leq p-1$. Suppose that $%
e_{1}=d_{1}$. Since $$n\geq k+1=\frac{p^{2}+p}{2}=\frac{p+1}{2}\cdot p,$$ we
have $e_{1}=d_{1}\geq \frac{p+1}{2}$ and $2e_{1}\geq p+1$. We obtain a
contradiction since $$n=d_{1}p+d_{0}=e_{1}p+2e_{1}+e_{0}\geq d_{1}p+p+1$$ but $%
d_{0}\leq p-1$. It follows that $e_{1}<d_{1}\leq p-1$. Write $%
d_{1}=e_{1}+\delta $ for some integer $\delta \geq 1$. Then $$\frac{d_{1}}{%
e_{1}}=1+\frac{\delta }{e_{1}} \geq \frac{p-1}{p-2}$$ is equivalent to $$
\frac{\delta }{e_{1}}\geq \frac{1}{p-2}.$$ Thus, \eqref{ast} is
true if and only if $\frac{\delta }{e_{1}}\geq \frac{1}{p-2}$. This is the case since $%
e_{1}<d_{1}\leq p-1$ implies $e_{1}\leq p-2$ and $\delta \geq 1$. Note
further that if $\frac{d_{1}}{e_{1}}=\frac{p-1}{p-2}$, which implies $%
r\left( n,p\right) =r\left( m,p\right) $, then $e_{1}<d_{1}\leq p-1$ implies 
$d_{1}=p-1$ and $e_{1}=p-2$. Therefore, $$%
n=e_{1}p+2e_{1}+e_{0}=p^{2}-4+e_{0}=m+e_{0}.$$ Since $e_{0}\geq 0$, we have,
in this case, $m\leq n$.

\item[(2).]  Suppose that $n$ is a $3$-digit number in base $p$, or in other words, $%
d_{2}\geq 1$. Before proceeding to prove the main claim of the proposition,
note that its last implication is vacuously true since $m$ is a $2$-digit
number in base $p$.

By the assumption $n<4p^{2}$ we get that $d_{2}\leq 3$. Furthermore%
\begin{eqnarray*}
n &=&d_{2}p^{2}+d_{1}p+d_{0}=e_{2}\left( p+2\right) ^{2}+e_{1}\left(
p+2\right) +e_{0} \\
&=&e_{2}p^{2}+\left( 4e_{2}+e_{1}\right) p+4e_{2}+2e_{1}+e_{0}.
\end{eqnarray*}%
Since $d_{1}p+d_{0}<p^{2}$, we must have $e_{2}\leq d_{2}\leq p-1$. We
consider the following cases.

\item[(2.1).] Suppose that $e_{2}=d_{2}$. In this case 
\begin{equation}\label{astast}
d_{1}p+d_{0}=\left( 4e_{2}+e_{1}\right)
p+4e_{2}+2e_{1}+e_{0}.
\end{equation}
Since $d_{0}\leq p-1$, we have $4e_{2}+e_{1}\leq d_{1}$ which is equivalent
to $$e_{1}\leq d_{1}-4d_{2}\leq p-1-4d_{2}.$$ Since $e_{1}\geq 0$, this
implies that $d_{2}\leq \frac{p-1}{4}$. From $e_{1}\leq p-1-4d_{2}$ and $%
d_{2}\geq 1$, we get $e_{1}\leq p-5$ and hence 
\begin{eqnarray*}
4e_{2}+2e_{1}+e_{0} &=&4d_{2}+2e_{1}+e_{0}\leq 4\cdot \frac{p-1}{4}+2p-10+p+1
\\
&=&p-1+2p-10+p+1=4p-10.
\end{eqnarray*}%
Thus, dividing $4e_{2}+2e_{1}+e_{0}$ by $p\geq 11$ with remainder, we have 
\begin{equation}
4e_{2}+2e_{1}+e_{0}=\alpha \cdot p+\beta \text{; }\alpha \in \left\{
0,1,2,3\right\} \text{, }\beta \in \left\{ 0,\ldots ,p-1\right\} .
\label{Eq_17}
\end{equation}%
Rewriting \eqref{astast} in terms of $\alpha $ and $\beta $ we
get:%
\begin{equation*}
d_{1}p+d_{0}=\left( 4e_{2}+e_{1}+\alpha \right) p+\beta .
\end{equation*}%
Hence $\beta =d_{0}$ and $d_{1}=4e_{2}+e_{1}+\alpha $. Since we assume $%
e_{2}=d_{2}$, Condition \eqref{ast} reads%
\begin{gather*}
\frac{d_{2}\left( p+1\right) +2d_{2}+e_{1}+2d_{2}+\alpha }{d_{2}\left(
p+1\right) +2d_{2}+e_{1}}\geq \frac{p-1}{p-2}. \\
\Longleftrightarrow \\
1+\frac{2d_{2}+\alpha }{d_{2}\left( p+1\right) +2d_{2}+e_{1}}\geq 1+\frac{1}{%
p-2}
\end{gather*}%
\begin{gather*}
\Longleftrightarrow \\
\left( 2d_{2}+\alpha \right) \left( p-2\right) \geq d_{2}\left( p+1\right)
+2d_{2}+e_{1} \\
\Longleftrightarrow \\
d_{2}\left( p-7\right) +\alpha \left( p-2\right) -e_{1}\geq 0
\end{gather*}%
If $\alpha >0$, we see, using $e_{1}\leq p-5$ and $p\geq 11$, that the left-hand side is strictly positive. If $\alpha =0$, we have to prove that 
\begin{equation*}
d_{2}\left( p-7\right) -e_{1}\geq 0.
\end{equation*}%
By (\ref{Eq_17}), $$4e_{2}+2e_{1}+e_{0}=4d_{2}+2e_{1}+e_{0}=\beta \leq p-1,$$
or, equivalently%
\begin{equation*}
2e_{1}\leq p-1-4d_{2}-e_{0}\leq p-1-4d_{2}.
\end{equation*}%
Since $d_{2}\geq 1$, this implies $e_{1}\leq \frac{p-5}{2}$. It follows that 
\begin{gather*}
d_{2}\left( p-7\right) -e_{1}\geq d_{2}\left( p-7\right) -\frac{p-5}{2}
\geq \left( p-7\right) -\frac{p-5}{2}=\frac{p-9}{2}>0
\end{gather*}%
as required. This concludes the proof that Condition \eqref{ast}
is true in case (2.1).

\item[(2.2).] Suppose that $e_{2}<d_{2}$. Since $n<4p^{2}$ we have $e_{2}<d_{2}\leq
3$, so $e_{2}\leq 2$.
As before, we start from%
\begin{eqnarray*}
n &=&d_{2}p^{2}+d_{1}p+d_{0}=e_{2}\left( p+2\right) ^{2}+e_{1}\left(
p+2\right) +e_{0} \\
&=&e_{2}p^{2}+\left( 4e_{2}+e_{1}\right) p+4e_{2}+2e_{1}+e_{0}\text{,}
\end{eqnarray*}%
and we write the relations between the two expansions using division with
remainder. Dividing $4e_{2}+2e_{1}+e_{0}$ by $p$ with remainder gives%
\begin{equation}
4e_{2}+2e_{1}+e_{0}=\alpha _{0}\cdot p+\beta _{0}\text{ where }0\leq \beta
_{0}\leq p-1.  \label{(30)}
\end{equation}%
Now we can rewrite 
\begin{equation*}
n=e_{2}p^{2}+\left( 4e_{2}+e_{1}+\alpha _{0}\right) p+\beta _{0}\text{,}
\end{equation*}%
and divide $4e_{2}+e_{1}+\alpha _{0}$ by $p$ with remainder:%
\begin{equation}
4e_{2}+e_{1}+\alpha _{0}=\alpha _{1}\cdot p+\beta _{1}\text{ where }0\leq
\beta _{1}\leq p-1.  \label{(31)}
\end{equation}%
Thus: 
\begin{equation*}
n=\left( e_{2}+\alpha _{1}\right) p^{2}+\beta _{1}p+\beta _{0}\text{,}
\end{equation*}%
and comparing with $n=d_{2}p^{2}+d_{1}p+d_{0}$ gives%
\begin{equation*}
d_{0}=\beta _{0}\text{, }d_{1}=\beta _{1}\text{, }d_{2}=e_{2}+\alpha _{1}%
.
\end{equation*}%
Note that the assumption $e_{2}<d_{2}$ implies $\alpha _{1}\geq 1$.

\item[(2.2.1).] Suppose that $e_{2}=0$. Then $d_{2}=\alpha _{1}$, and Condition \eqref{ast1} reads:
\begin{equation*}
\frac{\alpha _{1}p+d_{1}+\alpha _{1}}{e_{1}}\geq \frac{p-1}{p-2}.
\end{equation*}%
By (\ref{(31)}), $$4e_{2}+e_{1}+\alpha _{0}=\alpha _{1}\cdot p+\beta
_{1}=\alpha _{1}\cdot p+d_{1}.$$ Substituting $\alpha _{1}\cdot
p+d_{1}=e_{1}+\alpha _{0}$ into the last inequality gives: 
\begin{gather*}
\frac{e_{1}+\alpha _{0}+\alpha _{1}}{e_{1}}\geq 1+\frac{1}{p-2}%
\Longleftrightarrow \frac{\alpha _{0}+\alpha _{1}}{e_{1}}\geq \frac{1}{p-2}
\Longleftrightarrow 
\left( \alpha _{0}+\alpha _{1}\right) \left( p-2\right) \geq e_{1}
\end{gather*}%
Since $e_{1}\leq p+1$, the last inequality is true if $\alpha _{0}+\alpha
_{1}\geq 2$. By our assumptions, $\alpha _{1}=d_{2}\geq 1$, so we are done
if $\alpha _{0}\geq 1$. Otherwise, $\alpha _{0}=0$. In this case, since $$
4e_{2}+2e_{1}+e_{0}=\alpha _{0}\cdot p+\beta _{0,}$$ we get $$
2e_{1}+e_{0}=\beta _{0}\leq p-1$$ which implies $e_{1}\leq \frac{p-1}{2},$
and $$\left( \alpha _{0}+\alpha _{1}\right) \left( p-2\right) =\alpha
_{1}\left( p-2\right) \geq e_{1}$$ follows. Thus Condition \eqref{ast} is true in case (2.2.1).

\item[(2.2.2).] Suppose that $e_{2}=1$. Then $d_{2}=\alpha _{1}+1$, and Condition \eqref{ast} reads:%
\begin{gather*}
\frac{\left( \alpha _{1}+1\right) \left( p+1\right) +d_{1}}{p+3+e_{1}}\geq 
\frac{p-1}{p-2}
\Longleftrightarrow
\frac{p+1+\alpha _{1}+\alpha _{1}\cdot p+d_{1}}{p+3+e_{1}}\geq \frac{p-1}{p-2%
}.
\end{gather*}%
By (\ref{(31)}), $$\alpha _{1}\cdot p+d_{1}=4e_{2}+e_{1}+\alpha
_{0}=4+e_{1}+\alpha _{0},$$ we get the equivalent inequality:
\begin{align*}
\frac{p+3+e_{1}+\alpha _{1}+2+\alpha _{0}}{p+3+e_{1}}\geq \frac{p-1}{p-2}
&\Longleftrightarrow
\frac{\alpha _{1}+2+\alpha _{0}}{p+3+e_{1}}\geq \frac{1}{p-2} \\
&\Longleftrightarrow 
\left( \alpha _{1}+\alpha _{0}+2\right) \left( p-2\right) \geq p+3+e_{1}.
\end{align*}
Since $e_{1}\leq p+1$, the last inequality certainly holds true if 
\begin{equation}
\left( \alpha _{1}+\alpha _{0}+2\right) \left( p-2\right) \geq 2p+4.
\label{(32)}
\end{equation}%
Since 
$d_{2}=\alpha _{1}+1$ and $d_{2}>e_{2}=1$, we have $d_{2}=\alpha
_{1}+1\geq 2$, and hence $\alpha _{1}+2\geq 3.$ Therefore, $\alpha
_{1}+\alpha _{0}+2\geq 3$ and
\begin{equation*}
\left( \alpha _{1}+\alpha _{0}+2\right) \left( p-2\right) \geq 3\left(
p-2\right) .
\end{equation*}%
Since $3\left( p-2\right) \geq 2p+4$ is equivalent to $p\geq 10$ and the
proposition assumes $p\geq 11$, the proof of (\ref{(32)}) is done, and this
concludes the proof that Condition \eqref{ast} holds true in case
(2.2.2).

\item[(2.2.3).] Suppose that $e_{2}=2$. Then $d_{2}=\alpha _{1}+2$, and Condition \eqref{ast} reads:
\begin{gather*}
\frac{\left( \alpha _{1}+2\right) \left( p+1\right) +d_{1}}{2\left(
p+3\right) +e_{1}}\geq \frac{p-1}{p-2}
\Longleftrightarrow
\frac{2\left( p+1\right) +\alpha _{1}+\alpha _{1}\cdot p+d_{1}}{2\left(
p+3\right) +e_{1}}\geq \frac{p-1}{p-2}.
\end{gather*}%
Since $$\alpha _{1}\cdot p+d_{1}=4e_{2}+e_{1}+\alpha _{0}=8+e_{1}+\alpha _{0},$$
we get the equivalent inequality:
\begin{align*}
\frac{2\left( p+3\right) +e_{1}+\alpha _{0}+\alpha _{1}+4}{2\left(
p+3\right) +e_{1}}\geq \frac{p-1}{p-2}
&\Longleftrightarrow 
\frac{\alpha _{0}+\alpha _{1}+4}{2\left( p+3\right) +e_{1}}\geq \frac{1}{p-2}
\\
&\Longleftrightarrow
\left( \alpha _{0}+\alpha _{1}+4\right) \left( p-2\right) \geq 2\left(
p+3\right) +e_{1}.
\end{align*}%
Since $e_{1}\leq p+1$, the last inequality certainly holds true if 
\begin{equation*}
\left( \alpha _{0}+\alpha _{1}+4\right) \left( p-2\right) \geq 3p+7.
\end{equation*}%
Since $3\geq d_{2}=$ $\alpha _{1}+2$ and $d_{2}>e_{2}=2$, we have $\alpha
_{1}+2=3$, and hence $\alpha _{1}+4=5$. Therefore $\alpha _{1}+\alpha
_{0}+4\geq 5$ and%
\begin{equation*}
\left( \alpha _{0}+\alpha _{1}+4\right) \left( p-2\right) \geq 5\left(
p-2\right) \geq 3p+7\text{,}
\end{equation*}%
when $p\geq \frac{17}{2}$. Thus, Condition \eqref{ast} holds true
in case (2.2.3).

\end{enumerate}
\end{proof}

\section*{Appendix: Sample GAP Code\label{Sect_GAP_Code}}

The following is a documented GAP program that can be used for producing
examples that illustrate some points discussed in the paper. The code
contains one helper function and a ``main'' script which begins with the 
initialization of the two input parameters $n_1$ and $n_2$.

% Configuração local para o código GAP respeitar as margens
\lstset{
    basicstyle=\ttfamily\small, % Fonte monoespaçada e ligeiramente menor
    breaklines=true,            % QUEBRA AS LINHAS AUTOMATICAMENTE
    breakatwhitespace=false,    % Quebra mesmo no meio de palavras se faltar espaço
    columns=fullflexible,
    frame=single,               % Opcional: adiciona uma moldura elegante ao redor
    framesep=5pt
}

\begin{lstlisting}
# The_n_factorial_prime_power_factors_order_GAP.txt -
#
# Let n >= 2 be an integer, and let 2 = p_1 < p_2 < p_3 ... < p_k be
# the first k prime numbers where p_k is the largest prime smaller or
# equal n (thus it is useful to view k as a function of n).
# Evidently, p_1 < ... < p_k are all the prime factors of n! and the
# multiplicity of the prime p_i in the prime factorization of n! is
# denoted e_i := \nu_{p_i}(n!).
#
# We can define a new order relation over the set {1,2,3,...,k}, using
# the usual order relation < between the prime powers (p_i)^e_i of n!
# in the following way: For all i \ne j in {1,2,3,...,k} i is smaller
# than j if and only if (p_i)^e_i < (p_j)^e_j.
#
# The input to the current program are two integers 2 <= n_1 <= n_2.
# For each integer n_1 <= n <= n_2, the program calculates the prime
# power factorization of n! and sorts the list [1,2,3,...,k(n)] increasingly
# with respect to the order relation defined by the resulting prime power
# factors of n! as described above. The program prints out the sorted
# [1,2,3,...,k(n)] for each n in the domain [n_1,n_2].
#
# Example: Suppose n = 10. The prime power factors of n! are given by:
# 10! = 2^8 * 3^4 * 5^2 * 7^1 = 256 * 81 * 25 * 7
# and the program will print [4,3,2,1] which is the result of sorting
# [1,2,3,4] with respect to the order relation defined by the prime powers
# 2^8 > 3^4 > 5^2 > 7^1.
#
# Remarks:
#
# 1. The size of the input parameter n_2, and the size of n_2 - n_1 will be
#    constrained by GAP's limitations and the user's system. The program
#    does not check its input. Having said that, the program has run
#    successfully on the input [n_1,n_2] = [2,1000].
# 2. The current implementation calls GAP's PrimePowersInt function in
#    order to obtain the prime power factorization of n!. This task can be
#    performed more efficiently using Legendre's formula, if one needs to
#    study larger values of n.
# 3. One can utilize the current program in order to verify the numerical
#    part of the proof of Corollary 1.3, by taking [n_1,n_2] = [2,20], and
#    checking by inspection that the last entry of the printed list
#    "one_to_k_ordered_by_prime_power_factors_of_n_factorial" equals
#    1 for all n \ne 3.
#

#############################################################################
# list_prime_power_factors function
#
# Helper function that returns the list [p_1^e_1, p_2^e_2, ..., p_k^e_k]
#############################################################################

list_prime_power_factors := function(n)
    local list_pi_ei, num_primes, out, i;

    list_pi_ei := PrimePowersInt(n); # [p_1,e_1,p_2,e_2,...,p_k,e_k]
    num_primes := Length(list_pi_ei)/2;
    out := [];

    for i in [1..num_primes] do
        Add(out, list_pi_ei[2*i-1]^list_pi_ei[2*i]);
    od;

    return out;
end;

#############################################################################
# main
#############################################################################

### Input ###
n_1 := 2;
n_2 := 20;
### End Input ###

Print("\n\nCalculating [1,2,...,k(n)] sorted with respect to the order defined \n");
Print("by the prime power factors of n! for all n in [", n_1, "..", n_2, "]\n");

for n in [n_1..n_2] do
    factorial_n := Factorial(n);
    prime_power_factors := list_prime_power_factors(factorial_n);
    len_prime_power_factors := Length(prime_power_factors); # this is k(n)

    one_to_k_naturally_ordered := [1..len_prime_power_factors];
    # p_1 = 2 < p_2 = 3 < p_3 = 5 < ...

    SortParallel(prime_power_factors, one_to_k_naturally_ordered);
    # SortParallel(list1, list2) sorts list1 in increasing order
    # and, in parallel, applies to list2 the same exchanges
    # that are applied to list1

    one_to_k_ordered_by_prime_power_factors_of_n_factorial := one_to_k_naturally_ordered;

    Print("\n\n------------------------------------------------------------------");
    Print("\n\nn = ", n);
    Print("\n\n[1,2,..., ", len_prime_power_factors, "] sorted by prime power factors of n!=\n\n", one_to_k_ordered_by_prime_power_factors_of_n_factorial);
od;

Print("\n\n");
\end{lstlisting}

\end{document}